\documentclass[reqno,a4paper,12pt]{amsart}
\usepackage{color}
\usepackage{amsmath,amssymb,amsfonts,amsthm,enumerate}
\usepackage{mathrsfs,accents}
\usepackage[T1]{fontenc}
\usepackage[active]{srcltx}
\usepackage{epsfig}
\usepackage{graphicx}

\numberwithin{equation}{section}
\newtheorem{theorem}{Theorem}[section]
\newtheorem{lemma}[theorem]{Lemma}

\newtheorem{remark}[theorem]{Remark}

\newtheorem{corollary}[theorem]{Corollary}

\newcommand{\R}{\mathbb{R}}

\newcommand{\Z}{\mathbb{Z}}

\newcommand{\Haus}[1]{{\mathscr H}^{#1}} % Misura di Hausdorff
\newcommand{\Leb}[1]{{\mathscr L}^{#1}} % Misura di Lebesgue
\newcommand{\Per}{{\sf P}} %Perimeter
\newcommand{\I}{{\sf I}}
\newcommand{\J}{{\sf J}}
\newcommand\uno{{\bf 1}} %Characteristic function

\newcommand{\mean}[1]{\,-\hskip-1.08em\int_{#1}} %averaged integral
\newcommand{\textmean}[1]{- \hskip-.9em \int_{#1}} %averaged integral in text

\newcommand{\res}{\mathop{\hbox{\vrule height 7pt width .5pt depth 0pt
\vrule height .5pt width 6pt depth 0pt}}\nolimits}

\begin{document}

\title[$BMO$-type norms related to the perimeter of sets]{
$BMO$-type norms related to the perimeter of sets}

\author[L.\ Ambrosio]{Luigi Ambrosio}
\address{Scuola Normale Superiore, Piazza Cavalieri 7, 56100 Pisa, Italy}
\email{l.ambrosio@sns.it}

\author[J.\ Bourgain]{Jean Bourgain}
\address{School of Mathematics, Institute for Advanced Study Princeton, NJ 08540, USA}
\email{bourgain@ias.edu}

\author[H.\ Brezis]{Haim Brezis}
\address{Rutgers University, Department of Mathematics, Hill center, Busch Campus, 110 Frelinghuysen Road, Piscataway, NJ 08854, USA}
\address{Technion,
Department of Mathematics, Haifa 32000, Israel. }
\email{brezis@math.rutgers.edu}

\author[A.\ Figalli]{Alessio Figalli}
\address{University of Texas at Austin, Mathematics Dept., 2515 Speedway Stop C1200,
Austin, Texas 78712-1202, USA}
\email{figalli@math.utexas.edu}

 \begin{abstract}
 In this paper we consider an isotropic variant of the $BMO$-type norm recently introduced
 in \cite{bbm}. We prove that, when considering characteristic functions of sets,
 this norm is related to the perimeter. A byproduct of our analysis is a new 
 characterization of the perimeter of sets in terms of this norm, independent of the theory of
distributions.
 \end{abstract}

\maketitle

\section{Introduction}

Let $Q=(0,1)^n$ be the unit cube in $\R^n$, $n>1$.
In a very recent paper \cite{bbm}, the second and third author, in collaboration with P. Mironescu, introduced a new function
space $B\subset L^1(Q)$ based on the following 
seminorm, inspired by the celebrated $BMO$ space of John-Nirenberg \cite{JN}:
$$
\|f\|_B:=\sup_{\epsilon\in (0,1)}[f]_\epsilon,
$$
where 
\begin{equation}\label{def[]eps}
[f]_\epsilon:=\epsilon^{n-1}\sup_{{\mathcal G}_\epsilon}\sum_{Q'\in\mathcal G_\epsilon}
\mean{Q'}\Bigl|f(x)-\mean{Q'} f\Bigr|\, dx,
\end{equation}
and ${\mathcal G}_\epsilon$ denotes a collection of disjoint $\epsilon$-cubes $Q'\subset Q$ with sides parallel
to the coordinate axes and cardinality not exceeding $\epsilon^{1-n}$;
the supremum in \eqref{def[]eps} is taken over all such collections.
In addition to $\|f\|_B$, it is also useful to consider its infinitesimal version, namely
$$
[f]:=\limsup_{\epsilon\downarrow 0}[f]_\epsilon\leq \|f\|_B,
$$
and the space $B_0:=\{f \in B:\,[f]=0\}$.

Their main motivation was the search of a space $X$ on the one hand sufficiently large to include $VMO$, $W^{1,1}$, and the 
fractional Sobolev spaces $W^{1/p,p}$ with $1<p<\infty$, on the other hand sufficiently small (i.e., with a sufficiently strong
seminorm) to provide the implication 
\begin{equation}
\label{eq:X1}
f\in X \text{ and $\Z$-valued}\qquad\Longrightarrow\qquad \text{$f=k\,\,\,\Leb{n}$-a.e. in $Q$, for some $k\in\Z$}
\end{equation}
an implication known to be true in the spaces $VMO$, $W^{1,1}$, and $W^{1/p,p}$. 

One of the main results in \cite{bbm} asserts that \eqref{eq:X1} holds with $X=B_0$.
The principal ingredient in the proof concerns the case $f=\uno_A$,
where $A\subset Q$ is measurable.
For such special functions it is proved in \cite{bbm} that 
\begin{equation}\label{eq:bbm2}
\biggl\| f-\mean{Q} f\biggr\|_{L^{n/(n-1)}(Q)}\leq C\,[f];
\end{equation}
here and in what follows we denote by $C$ a generic constant depending only on $n$.
Estimate \eqref{eq:bbm2} suggests a connection with Sobolev embeddings and isoperimetric inequalities; recall e.g. that
\begin{equation}
\label{eq:X2}
\biggl\| f-\mean{Q} f\biggr\|_{L^{n/(n-1)}(Q)}\leq C\,|Df|(Q)\qquad \forall\,f \in BV(Q).
\end{equation}
When $f=\uno_A$, \eqref{eq:X2} takes the form
\begin{equation}
\label{eq:X3}
\biggl\| \uno_A-\mean{Q} \uno_A\biggr\|_{L^{n/(n-1)}(Q)}\leq C\,\Per(A,Q),
\end{equation}
where $\Per(A,Q)$ denotes the perimeter of $A$ relative to $Q$.
Combining \eqref{eq:X3} with the obvious inequality
$$
\biggl\| \uno_A-\mean{Q} \uno_A\biggr\|_{L^{n/(n-1)}(Q)} \leq 2
$$
yields 
\begin{equation}
\label{eq:X4}
\biggl\| \uno_A-\mean{Q} \uno_A\biggr\|_{L^{n/(n-1)}(Q)}\leq C\,\min\{1,\Per(A,Q)\}.
\end{equation}
In view of \eqref{eq:bbm2} and \eqref{eq:X4}
it is natural to ask whether there exists a relationship between $[\uno_A]$ and $\min\{1,\Per(A,Q)\}$.
The aim of this paper is to answer positively to this question.

Since the concept of perimeter is isotropic, it is better to make also the main object of \cite{bbm} isotropic, by considering  
\begin{equation}\label{defIeps}
\I_\epsilon(f):=\epsilon^{n-1}\sup_{{\mathcal F}_\epsilon}\sum_{Q'\in\mathcal F_\epsilon}
\mean{Q'}\Bigl|f(x)-\mean{Q'} f\Bigr|\, dx,
\end{equation}
where ${\mathcal F}_\epsilon$ denotes a collection of disjoint $\epsilon$-cubes $Q'\subset\R^n$
with \emph{arbitrary} orientation and cardinality not exceeding $\epsilon^{1-n}$. 

The main result of this paper is the following:

\begin{theorem} \label{thm:mainintro} For any measurable set $A\subset\R^n$ one has
\begin{equation}\label{eq:mainintro}
\lim_{\epsilon\to 0}\I_\epsilon(\uno_A)=\frac 12\min\bigl\{1,\Per(A)\bigr\}.
\end{equation}
In particular, $\lim_\epsilon\I_\epsilon(\uno_A)< 1/2$ implies that $A$ has finite perimeter and $\Per(A)=2\lim_\epsilon\I_\epsilon(\uno_A)$. 
\end{theorem}

We present the proof of Theorem~\ref{thm:mainintro} in Section~\ref{sec:main}. Although we confine ourselves
to the most interesting case $n>1$ throughout this paper, we point out in Section~3.4 that Theorem~\ref{thm:mainintro}
still holds in the case $n=1$ and we present a brief proof; since in this case $\Per(A)$ is an integer, we infer that
$\lim_\epsilon\I_\epsilon(\uno_A)< 1/2$ implies that either $A$ or $\R\setminus A$ are Lebesgue negligible.

Returning to the case $n>1$, we can also understand better the role of the upper bound on cardinality with the formula
\begin{equation}\label{eq:june3}
\lim_{\epsilon\downarrow 0}
\sup_{{\mathcal F}_{\epsilon,M}} \epsilon^{n-1}\sum_{Q'\in\mathcal F_{\epsilon,M}}
2\mean{Q'}\Bigl|\uno_A(x)-\mean{Q'} \uno_A\Bigr|\, dx= \min\{M,\Per(A)\} \qquad \forall\,M>0,
\end{equation}
where ${\mathcal F}_{\epsilon,M}$ denotes a collection of $\epsilon$-cubes with arbitrary orientation
and cardinality not exceeding $M\epsilon^{1-n}$. 
The proof of \eqref{eq:june3} can be achieved by a scaling argument.
Indeed,
setting $\rho=M^{-1/(n-1)}$, it suffices to apply \eqref{eq:mainintro} to $\tilde A=\rho A$, noticing that
$$
\min\{1,\Per(\tilde{A})\}=\frac{1}{M}\min\{M,\Per(A)\}=\rho^{n-1}\min\{M,\Per(A)\},
$$
that $\epsilon^{1-n}=M(\epsilon/\rho)^{1-n}$, and finally 
that the transformation $x\mapsto x/\rho$ maps $\tilde A$ to $A$, as well as $\epsilon$-cubes to $\epsilon/\rho$-cubes.

In Section \ref{sect:XX} we return to the framework of measurable subsets $A$ of 
a Lipschitz domain $\Omega$, and we establish that
\begin{equation}
\label{eq:X5}
\lim_{\epsilon \to 0}\I_\epsilon(\uno_A,\Omega)=\frac12 \min\bigl\{1,\Per(A,\Omega)\bigr\},
\end{equation}
where $\I_\epsilon(\uno_A,\Omega)$ is a localized version of \eqref{defIeps} where we restrict the supremum over cubes contained in $\Omega$.

Also, going back to the setting of \cite{bbm}, in Section \ref{sect:XX} we prove that
$$
[\uno_A]\leq \frac12 \min\{1,\Per(A,Q)\} \leq C\,[\uno_A]
$$
for every measurable subset of $Q$.

Then, in Section \ref{sect:no constraint} we discuss how removing the bound on the cardinality allows us to obtain a new characterization 
both of sets of finite perimeter and of the perimeter, independent of the theory of distributions. In a somewhat different direction, see also
\cite{X1}, \cite[Corollary 3 and Equation (46)]{X2}, and \cite{X3}.

We conclude this introduction with a few more words on the strategy of proof. As illustrated in Remark~\ref{rem:simple_idea}, 
using the canonical decomposition in cubes, it is not too difficult to show the existence of dimensional constants $\xi_n,\eta_n>0$ satisfying
\begin{equation}\label{eq:trivial_idea}
\limsup_{\epsilon\downarrow 0}\I_\epsilon(\uno_A)<\xi_n\qquad\Longrightarrow\qquad\Per(A)\leq \eta_n
\limsup_{\epsilon\downarrow 0}\I_\epsilon(\uno_A)
\end{equation}
for any measurable set $A\Subset Q$.
This idea can be very much refined, leading to the proof of the inequality $\liminf_\epsilon\I_\epsilon(\uno_A)\geq 1/2$
whenever $P(A)=+\infty$. Since $\I_\epsilon(\uno_A)\leq 1/2$, see \eqref{eq:trivial} below, this proves our main result for sets of infinite perimeter.
For sets of finite perimeter, the inequality $\leq$ in \eqref{eq:mainintro} relies on the relative isoperimetric inequality in the cube 
with sharp constant, see \eqref{eq:Hadwiger} below, while the inequality $\geq$ relies on a blow-up argument.\\

This paper originated from a meeting in Naples in November 2013, dedicated to Carlo Sbordone's 65th birthday, where three of us
(LA, HB, and AF) met. On that occasion HB presented some results from \cite{bbm} and formulated a conjecture which
became the motivation and the main result of the present paper. For this and many other reasons, we are happy to dedicate this paper to
Carlo Sbordone.

\section{Notation and preliminary results}

Throughout this paper we assume $n\geq 2$.
We denote by $\#\,F$ the cardinality of a set $F$, by $A^c$ the complement of $A$, by $|A|$ the Lebesgue measure of a (Lebesgue) measurable set 
$A\subset\R^n$, by $\Haus{n-1}$ the Hausdorff $(n-1)$-dimensional measure.
For $\delta>0$, we say that $Q'$ is a $\delta$-cube if $Q'$ is a cube obtained by rotating and translating the standard $\delta$-cube
$(0,\delta)^n$. 

\subsection{$BV$ functions, sets of finite perimeter, and relative isoperimetric inequalities}

Given $\Omega\subset\R^n$ open and $f\in L^1_{\rm loc}(\Omega)$, we define
\begin{equation}\label{eq:defper}
|Df|(\Omega):=\sup\left\{\int f(x) \,{\rm div}\phi(x)\,dx:\ \phi\in C^1_c(\Omega;\R^n),\,\,|\phi|\leq 1\right\} .
\end{equation}
By construction, $f\mapsto |Df|(\Omega)\in [0,\infty]$ is lower semicontinuous w.r.t. the
$L^1_{\rm loc}(\Omega)$ convergence. By Riesz theorem, whenever $|Df|(\Omega)$ is
finite the distributional derivative $Df=(D_1f,\ldots,D_nf)$ of $f$ is a vector-valued measure with finite
total variation, therefore $f\in BV_{\rm loc}(\Omega)$.\footnote{Recall that $f \in BV(U)$
if $f \in L^1(U)$ and $|Df|(U)<\infty$. } In addition, the total variation of $Df$
coincides with the supremum in \eqref{eq:defper} (thus, justifying our notation). 

We will mostly apply these concepts when $f=\uno_A$ is a characteristic function
of a measurable set $A\subset \R^n$. In this case we use the
traditional and more convenient notation
$$
\Per(A,\Omega)=|D\uno_A|(\Omega),\qquad\quad \Per(A)=\Per(A,\R^n).
$$
A key property of the perimeter is the so-called relative isoperimetric inequality: for any bounded open set $\Omega\subset\R^n$
with Lipschitz boundary one has
$$
|E|\cdot |\Omega\setminus E|\leq c(\Omega)\Per(E,\Omega)\quad\text{for any measurable set $E\subset\Omega$.}
$$
In the case when $\Omega$ is the unit cube $Q$, we will need the inequality with sharp constant:
\begin{equation}\label{eq:Hadwiger}
|E|(1-|E|)\leq \frac 14\Per(E,Q) \quad\text{for any measurable set $E\subset Q$.}
\end{equation}
This inequality is originally due to H. Hadwiger \cite{Ha} for polyhedral subsets of the cube. Far reaching variants
appeared subsequently in the literature (see e.g. S.G. Bobkov \cite{bo1,bo2}, D. Bakry and M. Ledoux \cite{bl}, F. Barthe and
B. Maurey \cite{bm}, and their references). However we could not find \eqref{eq:Hadwiger} stated in the required generality
used here (it is often formulated with the Minkowski content instead of the perimeter, so that some extra approximation
argument is anyhow needed). For this reason, and for the
reader's convenience, we have included in the appendix a proof of \eqref{eq:Hadwiger} based on the results
of \cite{bm}, in any number of space dimensions.

\subsection{Fine properties of sets of finite perimeter}

In \S\ref{sec:lbound} we will need finer properties of sets of finite perimeter $A$ in an open set $\Omega$. 
In \cite{dg1}, De Giorgi singled out a set ${\mathcal F}A$ of finite $\Haus{n-1}$-measure, called reduced boundary, on which
$|D\uno_A|$ is concentrated and $A$ is asymptotically close to a half-space. More precisely $|D\uno_A|=\Haus{n-1}\res {\mathcal F}A$, i.e.,
$|D\uno_A|(E)=\Haus{n-1}(E\cap {\mathcal F}A)$ for any Borel set $E\subset\Omega$. 
De Giorgi also proved that $|D\uno_A|$-almost all of the reduced boundary can be covered
by a sequence of $C^1$ hypersurfaces $\Gamma_i$ (the so-called rectifiability property). A few years later, Federer 
in \cite{Fed1} extended
these results to the so-called essential boundary, namely the complement of density 0 and density 1 sets:
\begin{equation}\label{eq:defbf}
\partial^* A:=\left\{x\in\R^n:\ \limsup_{r\to 0^+}\frac{|B_r(x)\cap A|}{|B_r(x)|}>0\quad \text{and}\quad
\limsup_{r\to 0^+}\frac{|B_r(x)\setminus A|}{|B_r(x)|}>0\right\}.
\end{equation}
Federer also slightly strenghtned the rectifiability result, by replacing $|D\uno_A|$ with $\Haus{n-1}$. 
We collect in the next theorem the results we need on sets of finite perimeter.

\begin{theorem} [De Giorgi-Federer]\label{thm:dgfe}
Let $A$ be a set of finite perimeter in $\Omega$. Then the following properties hold:
\begin{itemize}
\item[(i)]
\begin{equation}\label{eq:repPer}
|D\uno_A|(E)=\Haus{n-1}(E\cap \partial^*A)\qquad\text{for any Borel set $E\subset\Omega$;}
\end{equation}
\item[(ii)] there exist embedded $C^1$ hypersurfaces $\Gamma_i\subset\R^n$ satisfying
\begin{equation}\label{eq:rettibou}
\Haus{n-1}\biggl(\Omega\cap\partial^* A\setminus\bigcup_{i=1}^\infty\Gamma_i\biggr)=0;
\end{equation}
\item[(iii)] if $\Gamma_i$ are as in \eqref{eq:rettibou}, for $\Haus{n-1}$-a.e. $x\in\Omega\cap\partial^* A\cap\Gamma_i$ 
there exists a half-space $H_A(x)$ with inner normal $\nu_A(x)$ orthogonal to $\Gamma_i$ at $x$ such that
$\uno_{(A-x)/r} \to \uno_{H_A(x)}$ in $L^1_{\rm loc}(\R^n)$ as $r\rightarrow 0^+$;
\item[(iv)] if $F$ is any other set with finite perimeter in $\Omega$, 
$H_A(x)=\pm H_F(x)$ $\Haus{n-1}$-a.e. in $\Omega\cap\partial^*A\cap\partial^* F$. 
\end{itemize}
\end{theorem}
\begin{proof} For the first three properties, see \cite{dg1}, \cite{Fed1}, or 
\cite[Theorems~3.59 and 3.61]{afp}. Taking (iii) into account, the last assertion (iv) follows from the elementary property
$$
{\rm Tan}(\Gamma,x)={\rm Tan}(\tilde\Gamma,x)\qquad\text{for $\Haus{n-1}$-a.e. $x\in\Gamma\cap\tilde\Gamma$}
$$
whenever $\Gamma$ and $\tilde\Gamma$ are $C^1$ embedded hypersurfaces.
\end{proof}

Since $BV_{\rm loc}(\Omega)\cap L^\infty(\Omega)$ is easily seen to be an algebra with
$$
|D(uv)|\leq \|u\|_\infty|Du|+\|v\|_\infty|Dv|,
$$
it turns out that the class of sets of finite perimeter in an open set $\Omega$ is stable under relative 
complement, union, and intersection. We need also the following property:
\begin{equation}\label{eq:symmbou} 
\Haus{n-1}\bigl(\Omega\cap\partial^*(E\Delta F)\setminus (\partial^*E\Delta\partial^*F)\bigr)=0
\quad\text{whenever $E,F$ have finite perimeter in $\Omega$.}
\end{equation}
In order to prove it, we first notice that $\partial^*$ is invariant under complement and 
$\partial^*(E\cup F)\subset\partial^* E\cup\partial^* F$, 
$\partial^*(E\cap F)\subset\partial^* E\cup\partial^* F$, hence
it follows that $\partial^*(E\Delta F)\subset\partial^* E\cup\partial^* F$.
Then, take $x\in\Omega\cap\partial^*(E\Delta F)$
and assume (possibly permuting $E$ and $F$) that $x\in\partial^* E$. By property (iii) of Theorem~\ref{thm:dgfe}, possibly
ignoring a $\Haus{n-1}$-negligible set, we can also assume that $(E-x)/r$ converges as $r\rightarrow 0^+$ to a half-space
$H_E(x)$. Still ignoring another $\Haus{n-1}$-negligible set, we have then three possibilities for $F$: either $x$ is a point of density
1, or a point of density $0$, or there exists a half-space $H_F(x)$ such that $(F-x)/r\to H_F(x)$ as $r\rightarrow 0^+$.
In the first two cases it is clear that $x\in\partial^* E\setminus\partial^* F$ and we are done. In the third case, we know by property
(iv) of Theorem~\ref{thm:dgfe} that $H_E(x)=\pm H_F(x)$ for $\Haus{n-1}$-a.e. $x\in\Omega\cap\partial^*E\cap\partial^* F$. But
$H_E(x)=H_F(x)$ implies that $x$ is a point of density 0 for $E\Delta F$ and
$H_E(x)=-H_F(x)$ implies that $x$ is a point of density 1 for $E\Delta F$, so the third case can occur only on a 
$\Haus{n-1}$-negligible set.

\section{Proof of Theorem \ref{thm:mainintro} }\label{sec:main}

%\begin{theorem} Let $f=\uno_A$, with $A\subset\R^n$ measurable. Then
%\begin{equation}\label{eq:main}
%\lim_{\epsilon\to 0^+}\I_\epsilon(f)=\frac 1 2\min\left\{1,\Per(A)\right\}.
%\end{equation}
%\end{theorem}

The proof of Theorem \ref{thm:mainintro} is quite involved and will take all of this section.
Notice that since
\begin{equation}\label{eq:trivial}
\mean{Q'}\Bigl|\uno_A(x)-\mean{Q'}\uno_A\Bigr|\,dx=2\frac{|Q'\cap A|\cdot |Q'\setminus A|}{|Q'|^2}\leq\frac 12
\end{equation}
for any $\epsilon$-cube $Q'$,
we clearly have $\I_\epsilon(\uno_A)\leq 1/2$.

We now prove the theorem is three steps: first we show that $\I_\epsilon(\uno_A)\leq\Per(A)/2$
for all $\epsilon >0$, which proves that $\I_\epsilon(f)\leq \frac 1 2\min\left\{1,\Per(A)\right\}$.
Then we prove that $\liminf_{\epsilon\to 0^+}\I_\epsilon(f)\geq \frac 1 2\min\left\{1,\Per(A)\right\}$
first when $\Per(A)=\infty$ (the non-rectifiable case) and finally when
$A$ has finite perimeter (the rectifiable case).

\subsection{Upper bound}\label{sec:upperbound}

We prove that $\I_\epsilon(\uno_A)\leq\Per(A)/2$ for all $\epsilon>0$. For this, we may obviously assume $\Per(A)<\infty$,
hence $f=\uno_A\in BV_{\rm loc}(\R^n)$. By the additivity of $\Per(A,\cdot)$, it suffices to show that if $Q'$ is an 
$\epsilon$-cube, then
$$
\frac{|Q'\cap A|\cdot |Q'\setminus A|}{|Q'|^2}\leq \frac 14 \epsilon^{1-n}\Per(A,Q').
$$
After rescaling, this inequality reduces to \eqref{eq:Hadwiger}, which proves the desired result. 

\subsection{Lower bound: the non-rectifiable case}
\label{sect:unrectif}

Here we assume that $\Per(A)=\infty$ and we prove, under this assumption, that $\liminf_\epsilon\I_\epsilon(f)\geq 1/2$.
Before coming to the actual proof we sketch in the next remark the proof of \eqref{eq:trivial_idea}, announced in
the introduction.

\begin{remark}\label{rem:simple_idea} {\rm Let us consider the canonical subdivision (up to a Lebesgue negligible set) 
of $(0,1)^n$ in $2^{hn}$ cubes $Q_{i,h}$ with length side $2^{-h}$. We define on the scale $\epsilon=2^{-h}$ an approximate interior ${\rm Int}_h(A)$ of $A$
by considering the set
$$
I_h:=\biggl\{i\in\{1,\ldots,2^{hn}\}:\ \mean{Q_{i,h}}\uno_A>\frac 34\biggr\}
$$
and taking the union of the cubes $Q_{i,h}$, $i\in I_h$. Analogously we define a set of indices $E_h$ and the
corresponding approximate exterior ${\rm Ext}_h(A)={\rm Int}_h(Q\setminus A)$. We denote by
$F_h$ the complement of $I_h\cup E_h$ and by ${\rm Bdry}_h(Q)$ the union of the corresponding cubes.

Since ${\rm Int}_h(A)\to A$ in $L^1_{\rm loc}$ as $h \to \infty$, by the lower semicontinuity of the perimeter it suffices to
give a uniform estimate on $\Per({\rm Int}_h(A))=\Haus{n-1}(\partial \,{\rm Int}_h(A))$ as $h\to\infty$ under a smallness
assumption on $\limsup_h\I_{2^{-h}}(\uno_A)$.

Since $\textmean{Q_{i,h}}|\uno_A(x)-\textmean{Q_{i,h}}\uno_A|\,dx\geq 1/4$ for all $i\in F_h$ (by definition of $F_h$), we obtain that
\begin{equation}\label{eq:trivial_idea1}
\I_{2^{-h}}(\uno_A)<\frac{1}{4}\qquad\Longrightarrow\qquad \#\,F_h
\leq 4 \,\I_{2^{-h}}(\uno_A)\,(2^{-h})^{1-n}<(2^{-h})^{1-n},
\end{equation}
which provides a uniform estimate on $\Haus{n-1}(\partial\, {\rm Bdry}_h(A))$. Hence, to control
$\Haus{n-1}(\partial \,{\rm Int}_h(A))$ it
suffices to bound the number of faces $F\subset Q$ common to a cube $Q_{i,h}$ and
a cube $Q_{j,h}$, with $i\in I_h$ and $j\in E_h$. For this, notice that if $\tilde{Q}$ is any cube with side length $2^{1-h}$ containing
$Q_{i,h}\cup Q_{j,h}$, it is easily seen that 
$$
\mean{\tilde{Q}}\Bigl|\uno_A(x)-\mean{\tilde{Q}}\uno_A\Bigr|\,dx\geq 2^{-1-n}
$$
and this leads once more to an estimate of the number of these cubes with $(2^{1-h})^{1-n}$ provided
$\I_{2^{1-h}}(\uno_A)<2^{-1-n}$.
Combining this estimate with the uniform estimate on $\Haus{n-1}(\partial \,{\rm Bdry}_h(A))$ leads to \eqref{eq:trivial_idea}. 
}\end{remark}

We now refine the strategy above to prove:

\begin{lemma} \label{lem1} Let $K>0$ and $A\subset\R^n$ measurable with $\Per(A)=\infty$. 
Then there exists $\delta_0=\delta_0(K,A)>0$ with the following property: for all
$\delta\in (0,\delta_0]$ it is possible to find a disjoint collection $\mathcal U_\delta$ of $\delta$-cubes
satisfying:
\begin{itemize}
\item[(a)]  $2^{-n-1}<|Q'\cap A|/|Q'|< 1-2^{-n-1}$ for all $Q'\in\mathcal U_\delta$;
\item[(b)] $\#\,\mathcal U_\delta>K\delta^{-n+1}$;
\item[(c)] if $\mathcal U_\delta=\{Q_\delta(x_i)\}_{i\in I}$, the homothetic cubes $\{Q_{2\delta}(x_i)\}_{i\in I}$ are pairwise disjoint.
\end{itemize}
\end{lemma}
\begin{proof} In this proof we tacitly assume that all cubes have sides parallel to a fixed system of coordinates. 
Partition canonically $\R^n$ in  a family $\{Q_i\}_{i\in\Z^n}$ of $\delta/2$-cubes and set
$$
{\mathcal V}_{\delta/2}:=\left\{Q_i:\ \frac{|Q_i\cap A|}{|Q_i|}>\frac 12\right\},\qquad
A_{\delta/2}:=\bigcup_{Q_i\in{\mathcal V}_{\delta/2}}Q_i .
$$
Since $A_{\delta/2}\to A$ locally in measure as $\delta\to 0^+$, it follows from the lower semicontinuity of $\Per$ that
$$
\liminf_{\delta\to 0^+}\Per(A_{\delta/2})\geq\Per(A)=\infty.
$$
We define $\delta_0=\delta_0(K,A)>0$ by requiring that $\Per(A_{\delta/2})>2^{2n+2}n K$ for all $\delta\in (0,\delta_0]$. 

Fixing now $\delta\in (0,\delta_0]$ and defining
$$
\tilde{\mathcal V}:=\left\{Q_i\in\mathcal V_{\delta/2}:\ \text{$\Haus{n-1}(\partial Q_i\cap \partial A_{\delta/2})>0$}\right\}  
$$
as the subset of ``boundary cubes'' (see Figure \ref{fig1}) we can estimate
$$
2^{2n+2}nK<\Per(A_{\delta/2})\leq 2n\biggl(\frac{\delta}{2}\biggr)^{n-1}\#\,\tilde{\mathcal V},
$$
so that
\begin{equation}\label{eq:3.4}
\#\,\tilde{\mathcal V}>8^n K\delta^{-n+1}.
\end{equation}

\begin{figure}[h]
\includegraphics[scale=0.4]{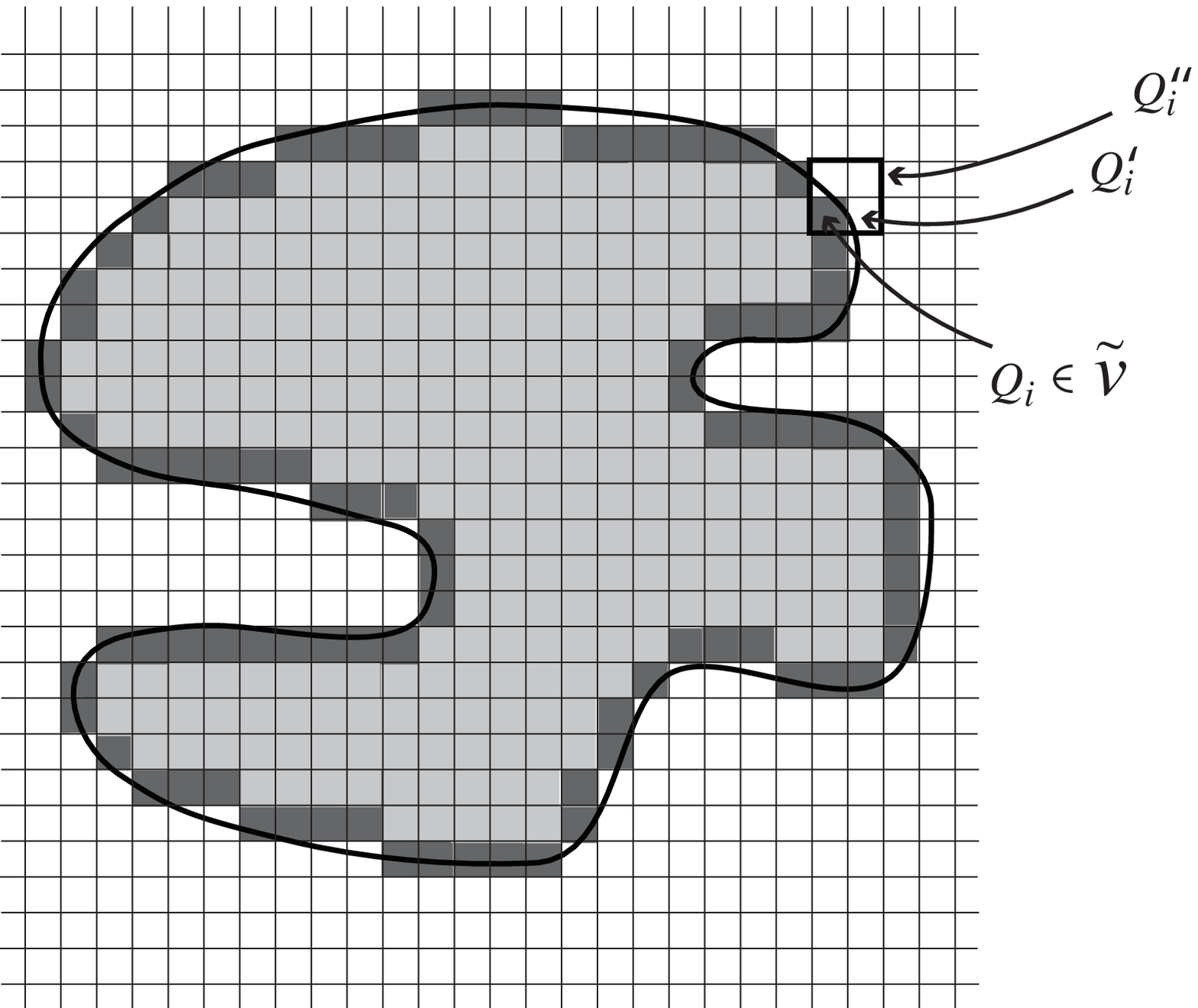}
\caption{\label{fig1}}
\end{figure}

Let $Q_i\in\tilde{\mathcal V}$ and let $Q_i'$ be a $\delta/2$-cube sharing a face $\sigma\subset\partial Q_i\cap\partial A_{\delta/2}$
with $Q_i$ (see Figure 1). Since obviously $Q_i'\notin\mathcal V_{\delta/2}$ we obtain $|Q_i'\cap A|\leq |Q_i'|/2$.  Hence, if $Q_i''$ is any $\delta$-cube containing
$Q_i\cup Q_i'$ we have 
$$
\begin{cases}
\displaystyle{|Q_i''\cap A|\geq |Q_i\cap A|>\frac 12 |Q_i|=\frac 1{2^{n+1}}|Q_i''|,}&\cr\cr
\displaystyle{|Q_i''\cap A^c|\geq |Q_i'\cap A^c|\geq \frac 12 |Q_i'|=\frac 1{2^{n+1}}|Q_i''|.}
\end{cases}
$$
It then suffices to consider a maximal subfamily ${\mathcal V}^*\subset\tilde{\mathcal V}$ of $\delta/2$ cubes with centers
at mutual distance (along at least  one of the coordinate directions) larger or equal than $7\delta/2$ and define
$$
\mathcal U_\delta:=\left\{Q_i'':\ Q_i\in{\mathcal V}^*\right\}.
$$
It is easy to check that $\mathcal U_\delta$ is a family of $\delta$-cubes whose homothetic enlargements by a factor 2 along their
centers are disjoint, so that (c) holds, and that (a) holds as well. In order to check (b), we notice
that the union of the enlargements by a factor $8$ of all cubes in ${\mathcal V}^*$ contains $\tilde{\mathcal V}$, by the maximality
of ${\mathcal V}^*$. Hence, from \eqref{eq:3.4} we get  
$$
\#\,{\mathcal V}^*\geq 8^{-n}\#\,\tilde{\mathcal V}>K\delta^{-n+1}.
$$
\end{proof}

\begin{lemma} \label{lem2} Let $c_0\in (0,1/2)$ and $A\subset (0,1)^n=Q$ measurable, with
\begin{equation}\label{eq:3.5}
c_0<|A|<1-c_0.
\end{equation}
Then, there exists $\epsilon_0=\epsilon_0(c_0,A)>0$ with the following property: for 
$\epsilon\in (0,\epsilon_0)$ there exists a disjoint collection ${\mathcal G}_\epsilon$ of $\epsilon$-cubes
contained in $(0,1)^n$ and satisfying
\begin{equation}\label{eq:3.6}
\frac{|V\cap A|}{|V|}=\frac 12\qquad\forall \,V\in\mathcal G_\epsilon,
\end{equation}
\begin{equation}\label{eq:3.7}
\#\,{\mathcal G}_\epsilon> c_1\epsilon^{-n+1},
\end{equation}
with $c_1>0$ depending only on $c_0$.
\end{lemma}
\begin{proof} 
First we choose $\epsilon_*=\epsilon_*(c_0,n)\in (0,1/2)$ such that the sets $A_1:=A\cap (\epsilon_*,1-\epsilon_*)^n\subset A$ and
$A_2:=A\cup \left[(0,1)^n\setminus (\epsilon_*,1-\epsilon_*)^n\right]\supset A$ satisfy
\begin{equation}\label{eq:3.5bis}
\frac{c_0}{2}<|A_1|,\qquad |A_2|<1-  \frac{c_0}{2}.
\end{equation}

We now extend the set $A_2$ by periodicity:
$$
\tilde A_2:=\bigcup_{h \in \mathbb Z^n} A_2+h, \qquad \tilde A_2^c=\R^n\setminus \tilde A_2.
$$
Then \eqref{eq:3.5bis} implies 
$$
\int_{Q}\int_{Q}\uno_{A_1}(x)\uno_{\tilde A_2^c}(x+z)\,dx\,dz=|A_1|(1-|A_2|)>\frac{c_0^2}{4}.
$$
Hence, we can find a nonzero vector $z\in Q$ satisfying
$$
|A_1\cap (\tilde A_2^c-z)|>\frac{c_0^2}{4}.
$$
Set now $e:=z/|z|$, $H_a:=z^\perp+ae$, $\hat A:=A_1\cap (\tilde A_2^c-z)$, and
$$
A_\delta:=\left\{x\in \hat A:\ \frac{|Q_r(y)\cap A_1|}{|Q_r(y)|}>\frac{1}{2}\quad \forall \,y\in Q_r(x),\,\,r\in (0,\delta)\right\}.
$$
Since $A_\delta$ monotonically converge as $\delta\downarrow 0$ to a set containing the set
of points of density 1 of $\hat A$, it follows that $|A_\delta|>c_0^2/4$ for $\delta$ small enough. Hence, because
$$
|A_\delta|\leq 2\int_{-\sqrt{n}/2}^{\sqrt{n}/2}\Haus{n-1}(A_\delta\cap H_a) \,da,
$$
we can find $a\in (-\sqrt{n}/2,\sqrt{n}/2)$ satisfying
\begin{equation}\label{eq:3.8}
\Haus{n-1}(A_\delta\cap H_a)>\frac{c_0^2}{8\sqrt{n}}.
\end{equation}

\begin{figure}[h]
\includegraphics[scale=0.4]{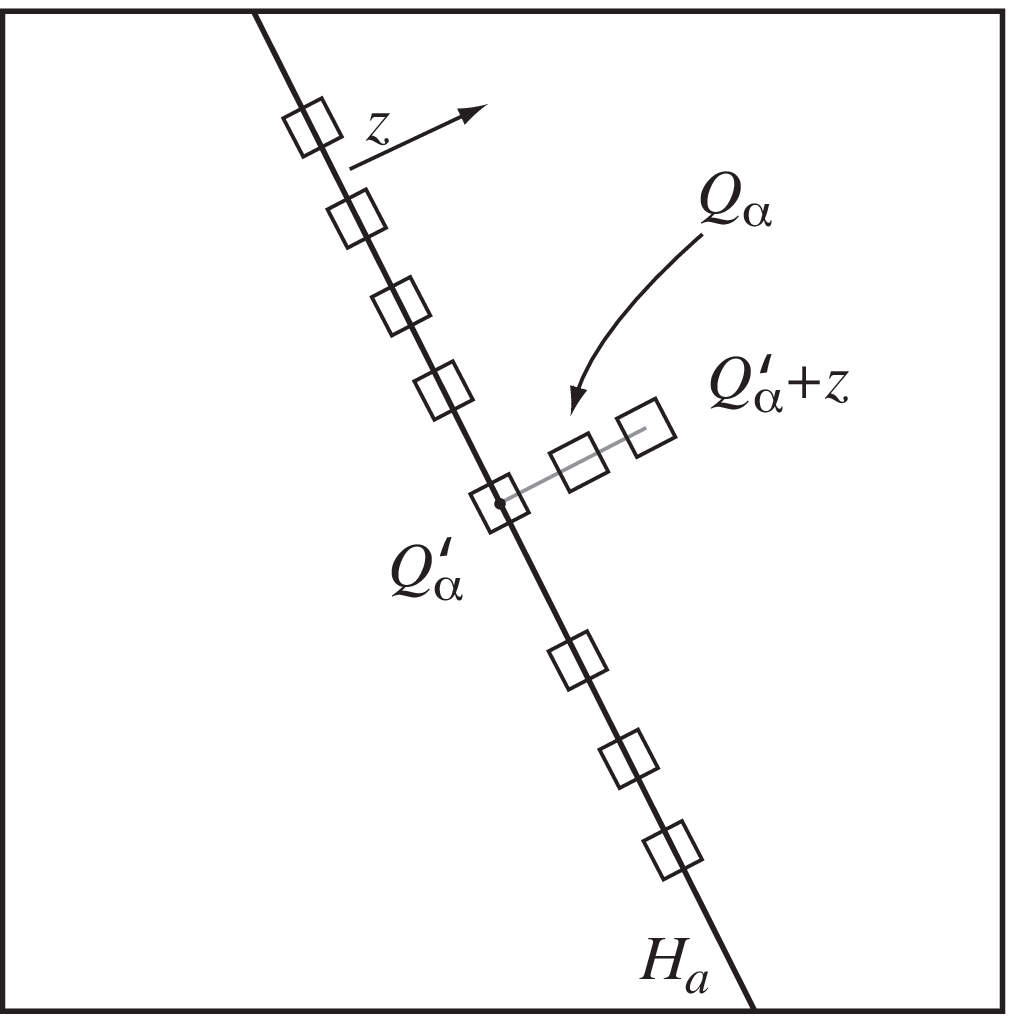}
\caption{\label{fig2}}
\end{figure}

For $\epsilon\leq\delta$, let us consider a canonical division $\{R_\alpha\}_{\alpha\in\Z^{n-1}}$
of $H_a$ in $\epsilon$-cubes of dimension $n-1$, and select those cubes $R_\alpha$ that satisfy
$\Haus{n-1}(A_\delta\cap R_\alpha)>\Haus{n-1}(R_\alpha)/2$, to build a family $\mathcal R_\epsilon$. Since  
$$
\bigcup_{R_\alpha\in \mathcal R_\epsilon}R_\alpha\to A_\delta\cap H_a \qquad\text{in $\Haus{n-1}$-measure
as $\epsilon\to 0^+$}, 
$$
we obtain from \eqref{eq:3.8}
\begin{equation}\label{eq:3.8bis}
\limsup_{\epsilon\rightarrow 0}\epsilon^{n-1}\#\,\mathcal R_\epsilon >\frac{c_0^2}{8\sqrt{n}}.
\end{equation}
Out of $\mathcal R_\epsilon$ we can build a disjoint collection $\mathcal G_\epsilon'$ of $\epsilon$-cubes $Q_\alpha'$ centered
at points $x_\alpha\in H_a$ with faces either orthogonal or parallel to $z$, such that
\begin{equation}\label{eq:3.9}
|Q_\alpha'\cap \hat A|>\frac 1 2|Q_\alpha'|,
\end{equation}
\begin{equation}\label{eq:3.10}
\#\,\mathcal G_\epsilon'>\frac{c_0^2}{8\sqrt{n}}\epsilon^{-n+1}.
\end{equation}
Indeed, \eqref{eq:3.9} follows from the definition of $A_\delta$, while \eqref{eq:3.10} follows by \eqref{eq:3.8bis}.
It follows from \eqref{eq:3.9} and the definition of $\hat A$ that
$$
|Q_\alpha'\cap A|\geq |Q_\alpha'\cap A_1|>\frac{1}{2}|Q_\alpha'|\quad\text{and}\quad 
|(Q_\alpha'+z)\cap A^c|\geq |(Q_\alpha'+z)\cap \tilde A_2^c|>\frac{1}{2}|Q_\alpha'|. 
$$
Since $A_1$ does not intersect $(0,1)^n\setminus (\epsilon_*,1-\epsilon_*)^n$, if $\epsilon<\epsilon^*/\sqrt{n}$
we obtain that $Q_\alpha'\subset (0,1)^n$.
Analogously, since $A_2$ contains 
$(0,1)^n\setminus (\epsilon_*,1-\epsilon_*)^n$, if $\epsilon<\epsilon^*/\sqrt{n}$ we obtain that
 $Q_\alpha'+z\cap \partial Q=\emptyset$, which implies that there exists a vector $h$ in $\mathbb Z^n$ such that $Q_\alpha'+z+h \subset (0,1)^n$ (to be precise, $h$ is of the form $-\gamma_1 e_1+\ldots-\gamma_n e_n$ with $\gamma_i \in \{0,1\}$).

Hence, by a continuity argument there exists $t_\alpha\in (0,1)$ such that, setting $Q_\alpha:=Q_\alpha'+t_\alpha (z+h)$, one has
$Q_\alpha\subset (0,1)^n$ and
$|Q_\alpha\cap A|=|Q_\alpha|/2$ (see Figure~\ref{fig2}, that corresponds to the case $h=0$). 
Then we can define $\mathcal G_\epsilon$ as the collection of the cubes $Q_\alpha$, which is
disjoint by construction (since their projections on $H_a$ are disjoint). 
\end{proof}

We can now prove that $\liminf_\epsilon\I_\epsilon(f)\geq 1/2$. Set $c_0=2^{-n-1}$, let $c_1$ be given by Lemma~\ref{lem2}, and set $K:=1/c_1$. 
If $\delta=\delta_0(A,K)$ is given by Lemma~\ref{lem1}, we can apply Lemma~\ref{lem1} to obtain a finite
disjoint family ${\mathcal U}_\delta$ of $\delta$-cubes with $\#\,{\mathcal U}_\delta>c_1^{-1}\delta^{1-n}$ and
$$
c_0<\frac{|Q'\cap A|}{|Q'|}< 1-c_0\qquad\text{for all $Q'\in\mathcal U_\delta$.}
$$
Since $\mathcal U_\delta$ is finite, for $0<\epsilon\ll\delta$ and all $Q'\in {\mathcal U}_\delta$ 
we can apply Lemma~\ref{lem2} to a rescaled copy by a factor $\delta^{-1}$ of $Q'$ and $A\cap Q'$ 
to obtain a disjoint family
${\mathcal G}_{\epsilon}(Q')$ of $\epsilon$-cubes contained in $Q'$ and satisfying
\begin{equation}\label{eq:3.11}
\frac{|V\cap A|}{|V|}=\frac 12\qquad\forall\, V\in{\mathcal G}_{\epsilon}(Q'),
\end{equation}
\begin{equation}\label{eq:3.12}
\#\,{\mathcal G}_{\epsilon}(Q')>c_1\biggl(\frac{\delta}{\epsilon}\biggr)^{n-1}.
\end{equation}
Now, by construction, the family 
$$
{\mathcal G}_{\epsilon}:=\bigcup_{Q'\in\mathcal U_\delta}{\mathcal G}_{\epsilon}(Q')
$$
of $\epsilon$-cubes is disjoint (taking into account condition (c) of Lemma~\ref{lem1}) 
and $|V\cap A|/|V|=1/2$ for each $V$ in the family. In addition, its cardinality can be estimated
from below as follows: 
$$
\#\,\mathcal G_\epsilon\geq c_1\biggl(\frac{\delta}{\epsilon}\biggr)^{n-1}\#\,\mathcal U_\delta>
c_1\biggl(\frac{\delta}{\epsilon}\biggr)^{n-1} \frac{1}{c_1}\delta ^{1-n}=\epsilon^{1-n}.
$$
Extracting from $\mathcal G_\epsilon$ a subfamily $\mathcal F_\epsilon$ with $\#\,\mathcal F_\epsilon= [\epsilon^{1-n}]$
we get
$$
\I_\epsilon(f)\geq \epsilon^{1-n}\sum_{V\in \mathcal F_\epsilon}\mean{V}\mean{V}|f(x)-f(y)|\,dx\,dy=
2\epsilon^{n-1}\sum_{V\in \mathcal F_\epsilon}\frac{|V\cap A|\cdot |V\setminus A|}{|V|^2}=
\frac{1}{2}\epsilon^{n-1}[\epsilon^{1-n}].
$$
By taking the limit as $\epsilon\to 0^+$ the conclusion is achieved.

\subsection{Lower bound: the rectifiable case}\label{sec:lbound}

The heuristic idea of the proof is to choose cubes well adapted to the local geometry of $\partial A$, as in Figure \ref{fig3} below.
Although it is easy to make this argument rigorous if $\partial A$ is smooth,
when $A$ has merely finite perimeter the argument becomes much less obvious.
Still,  the rectifiability of $\partial^* A$ and a suitable localization/blow-up argument allow us to prove the result in this general setting.

\begin{figure}[h]
\includegraphics[scale=0.3]{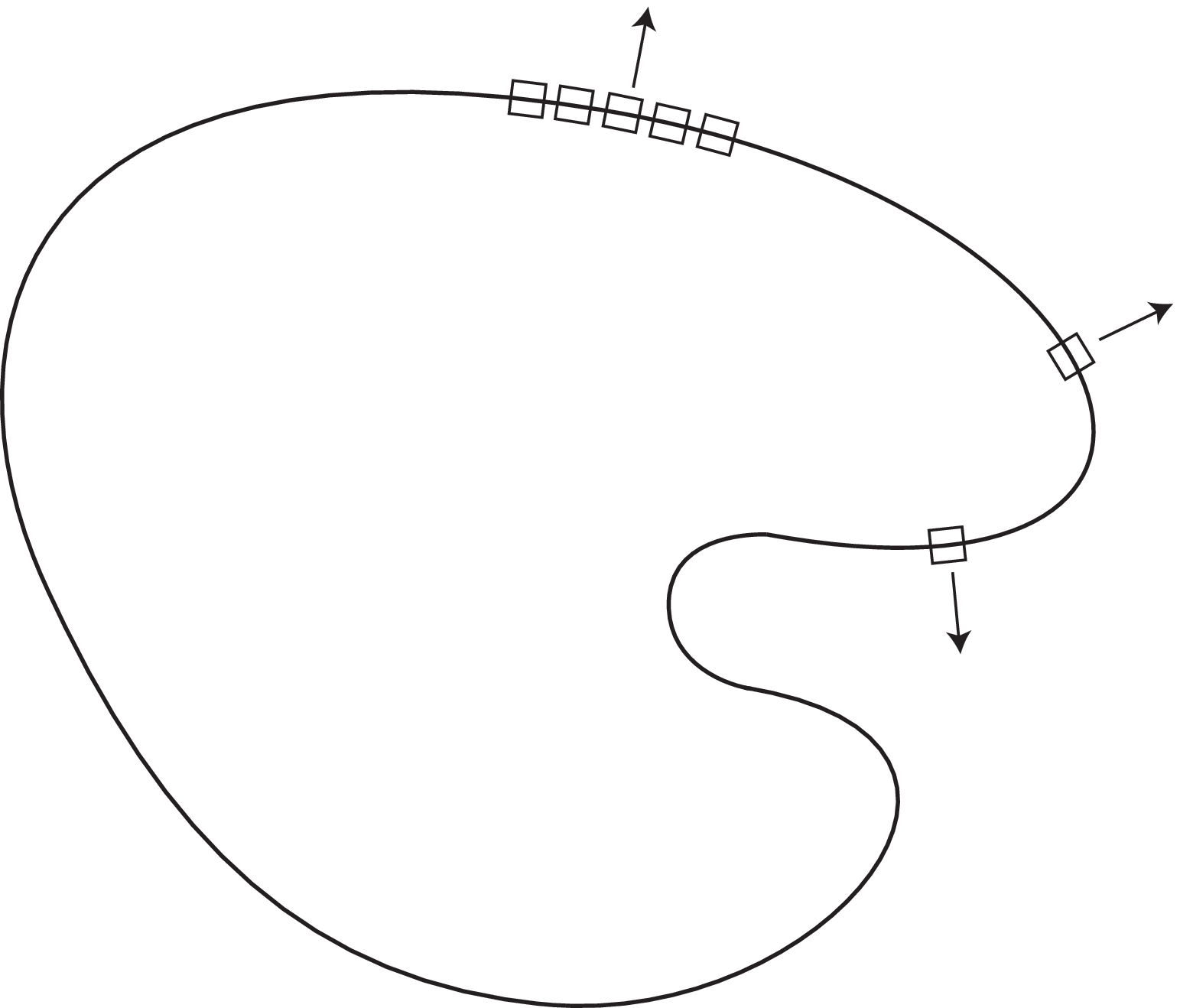}
\caption{\label{fig3}}
\end{figure}

Let $A\subset\R^n$ be measurable and $\Omega\subset\R^n$ open.
We localize $\I_\epsilon(\uno_A)$ to $\Omega$ 
and, at the same time, we impose a scale-invariant bound on the cardinality of the families by defining
$$
\J_\epsilon(A,\Omega):=\epsilon^{n-1}\sup_{{\mathcal F}_\epsilon}\sum_{Q'\in\mathcal F_\epsilon}
\mean{Q'}\Bigl|\uno_A(x)-\mean{Q'}\uno_A\Bigr|\, dx,
$$
where the supremum runs, this time, among all collections of disjoint families of $\epsilon$-cubes contained in $\Omega$,
with arbitrary orientation and cardinality not exceeding $\Per(A,\Omega)\epsilon^{1-n}$.

Notice that $\J_\epsilon$ has a nice scaling property, namely
\begin{equation}\label{eq:scalingJ}
\J_\epsilon(A/r,\Omega)=r^{1-n}\J_{r\epsilon}(A,r\Omega)\qquad\forall\, r>0.
\end{equation}
In addition, the additivity of $\Per(A,\cdot)$ shows that $\J_\epsilon$ is superadditive, namely
\begin{equation}\label{eq:superadd}
\J_\epsilon(A,\Omega_1\cup\Omega_2)\geq \J_\epsilon(A,\Omega_1)+\J_\epsilon(A,\Omega_2)
\qquad\text{whenever $\Omega_1\cap\Omega_2=\emptyset$.}
\end{equation}
Then, the lower bound 
\begin{equation}\label{eq:june6}
2\liminf_\epsilon\I_\epsilon(\uno_A)\geq\min\{1,\Per(A)\}
\end{equation}
 is a direct consequence
of Theorem~\ref{thm:limper} below, choosing $\Omega=\R^n$. Indeed, since $\Per(A)\leq 1$ implies $\J_\epsilon(A,\R^n)\leq\I_\epsilon(\uno_A)$
we obtain \eqref{eq:june6} when $\Per(A)\leq 1$. If $\Per(A)>1$, let $k=[\Per(A)]\geq 1$ be
its integer part and split any disjoint family $\mathcal F_\epsilon$ of $\epsilon$-cubes with maximal cardinality 
which enters in the definition of $\J_\epsilon(A,\R^n)$
into $k$ subfamilies with cardinality $[\epsilon^{1-n}]$ and a remainder subfamily of cardinality not exceeding 
$\Per(A)\epsilon^{1-n}-k[\epsilon^{1-n}]\leq \Per(A)\epsilon^{1-n}-k(\epsilon^{1-n}-1)$.
Since $\mathcal F_\epsilon$ is arbitrary, recalling \eqref{eq:trivial} we see that
$$
\J_\epsilon(A,\R^n)\leq k\,\I_\epsilon(\uno_A)+\frac 12\bigl(\Per(A)-k\bigr)+k\epsilon^{n-1}.
$$
Applying once more Theorem~\ref{thm:limper} with $\Omega=\R^n$ yields
$$
2\liminf_{\epsilon\to 0}\I_\epsilon(\uno_A)\geq \frac{\Per(A)}{k}-\frac{(\Per(A)-k)}{k}=
 1=\min\{1,\Per(A)\}
$$
since $\Per(A)>1$.

\begin{theorem} \label{thm:limper} 
For any measurable set $A\subset\R^n$ with finite perimeter in $\Omega$ one has
$$
\lim_{\epsilon\to 0^+}\J_\epsilon(A,\Omega)= \frac{1}{2}\Per(A,\Omega).
$$
\end{theorem}

The proof of the upper bound $\limsup_\epsilon\J_\epsilon(A,\Omega)\leq \Per(A,\Omega)/2$ can be
obtained exactly as in \S\ref{sec:upperbound}, so we focus on the lower bound. To this aim, 
it will be convenient to introduce the function
$$
\J_-(A,\Omega):=\liminf_{\epsilon\to 0^+}\J_\epsilon(A,\Omega).
$$
Because of \eqref{eq:scalingJ} we get
\begin{equation}\label{eq:scalingJ-bis}
\J_-(A/r,\Omega)=r^{1-n}\J_-(A,r\Omega)\qquad\forall r>0.
\end{equation}
In addition, the superadditivity of $\J_\epsilon(A,\cdot)$ and of the $\liminf$ give
\begin{equation}\label{eq:superadd-bis}
\J_-(A,\Omega_1\cup\Omega_2)\geq \J_-(A,\Omega_1)+\J_-(A,\Omega_2)
\qquad\text{whenever $\Omega_1\cap\Omega_2=\emptyset$.}
\end{equation}

In the first lemma we consider (local) subgraphs of $C^1$ functions.

\begin{lemma} \label{lem:c1blow} Let $E$ be the subgraph of a $C^1$ function in a neighbourhod of $0$. Then
$$
\liminf_{r\to 0^+}\J_-(E/r,B_1)\geq \frac 12\omega_{n-1}.
$$
\end{lemma}
\begin{proof} The proof is elementary, just choosing the canonical division in $\epsilon$-cubes, if
$B_r\cap\partial E$ is contained in a hyperplane for $r>0$ small enough. In the general case we
use the fact that $E/r$ is bi-Lipschitz equivalent to a half-space in $B_1$, with bi-Lipschitz constants
converging to $1$ as $r\downarrow 0$.
\end{proof}

In the second lemma we provide a sort of modulus of continuity for $E\mapsto\J_\epsilon (E,\Omega)$.

\begin{lemma} Let $E,\,F\subset\Omega$ be sets of finite perimeter in $\Omega$. Then 
\begin{equation}\label{eq:modulus_continuity}
\J_\epsilon(F,\Omega)\leq\J_\epsilon(E,\Omega)+\frac 12\epsilon^{n-1}+
\Haus{n-1}\bigl((\partial^*F\Delta\partial^*E)\cap\Omega\bigr)\qquad\forall\,\epsilon>0.
\end{equation}
\end{lemma}
\begin{proof} The inequality $\min\{z,1-z\}\leq 2z(1-z)$ in $[0,1]$ combined with \eqref{eq:Hadwiger} yields the relative isoperimetric inequality
\begin{equation}\label{eq:relative2}
\min\{|L|,|Q'\setminus L|\}\leq \frac{\epsilon}{2}\Per(L,Q')\qquad\text{for any $\epsilon$-cube $Q'$ with $L\subset Q'$.}
\end{equation}
Let now $\mathcal F_\epsilon$ be a family of $\epsilon$-cubes contained in $\Omega$ with cardinality less than
$\epsilon^{1-n}\Per(F,\Omega)$. For any $Q'\in\mathcal F_\epsilon$, adding and subtracting $\uno_E$
we have
$$
\mean{Q'}\mean{Q'}|\uno_F(x)-\uno_F(y)|\,dx\,dy\leq \mean{Q'}\mean{Q'}|\uno_E(x)-\uno_E(y)|\,dx\,dy+
\epsilon^{-n}|Q'\cap (F\Delta E)|.
$$
Analogously,  adding and subtracting
$\uno_{Q'\setminus E}$ and using $\uno_{E^c}(x)-\uno_{E^c}(y)=\uno_E(y)-\uno_E(x)$, we have
$$
\mean{Q'}\mean{Q'}|\uno_F(x)-\uno_F(y)|\,dx\,dy\leq \mean{Q'}\mean{Q'}|\uno_E(x)-\uno_E(y)|\,dx\,dy+
\epsilon^{-n}|Q'\cap (F\Delta E^c)|.
$$
Since $F\Delta E=\Omega\setminus (F\Delta E^c)$, we can apply \eqref{eq:relative2} with
$L=Q'\cap (F\Delta E)$, single out from
$\mathcal F_\epsilon$ a maximal subfamily 
with cardinality less than $\epsilon^{1-n}\Per(E,\Omega)$,  and use \eqref{eq:trivial}
and the
definition of $\J_\epsilon(E,\Omega)$ to get
\begin{eqnarray*}
\epsilon^{n-1}\sum_{Q'\in\mathcal F_\epsilon}\mean{Q'}\mean{Q'}|\uno_F(x)-\uno_F(y)|&\leq&\J_\epsilon(E,\Omega)+\frac 12\epsilon^{n-1}+
\frac 12\bigl(\Per(F,\Omega)-\Per(E,\Omega)\bigr)^+\\&+&\frac 12\sum_{Q'\in\mathcal F_\epsilon}\Per(F\Delta E,Q').
\end{eqnarray*}
Then, we use the additivity of $\Per(F\Delta E,\cdot)$ and take the supremum in the left hand side to obtain
$$
\J_\epsilon(F,\Omega)\leq\J_\epsilon(E,\Omega)+\frac 12\epsilon^{n-1}+\frac 12\bigl(\Per(F,\Omega)-\Per(E,\Omega)\bigr)^++\frac 12\Per(F\Delta E,\Omega),
$$
and we conclude using
 \eqref{eq:repPer} and \eqref{eq:symmbou}.
 \end{proof}

Notice that, in particular, the previous lemma gives
\begin{equation}\label{eq:modulus_continuitybis}
\J_-(F,\Omega)\leq\J_-(E,\Omega)+\Haus{n-1}\bigl((\partial^*F\Delta\partial^*E)\cap\Omega\bigr).
\end{equation}

In the third lemma we prove a density lower bound for $\J_-(E,\cdot)$ by comparing $E$ on small scales
with the subgraph of a $C^1$ function.

\begin{lemma} \label{lem:3} If $E$ has locally finite perimeter in $\Omega$, then
\begin{equation}\label{eq:finlim} 
\liminf_{r\to 0^+}\frac{\J_-(E,B_r(x))}{\omega_{n-1}r^{n-1}}\geq \frac 12
\qquad\text{for $|D\uno_E|$-a.e. $x\in\Omega$.}
\end{equation}
\end{lemma}
\begin{proof} Recall that $|D\uno_E|=\Haus{n-1}\res\partial^*E$ on Borel sets of $\Omega$.
In view of \eqref{eq:modulus_continuitybis}, the scaling property \eqref{eq:scalingJ-bis} of $\J_-$, and Lemma~\ref{lem:c1blow}, it
suffices to show that for $\Haus{n-1}$-a.e. $x\in\partial^*E$ there exists a set $F$ which is the subgraph of a $C^1$ function
in the neighbourhood of $x$, with
\begin{equation}\label{eq:osmall}
\Haus{n-1}\bigl((\partial^*F\Delta\partial^*E)\cap B_r(x))=o(r^{n-1}).
\end{equation}
To this aim, we use the representation \eqref{eq:rettibou}, we fix $i$ and consider a point $x\in\Gamma_i\cap\partial^* E$ where\footnote{Here we use (first with $S=\Gamma_i$, then with $S=\partial^*E$) 
the property that $\Haus{n-1}(S\cap B_r(x))=o(r^{n-1})$ for $\Haus{n-1}$-a.e. $x\in\R^n\setminus S$
whenever $S$ has locally finite $\Haus{n-1}$-measure, see for instance \cite[pag. 79, Eq. (2.41)]{afp}.}
$$
\Haus{n-1}\bigl((\Gamma_i\setminus \partial^*E)\cap B_r(x)\bigr)=o(r^{n-1})\quad\text{and}\quad
\Haus{n-1}\bigl((\partial^*E\setminus \Gamma_i)\cap B_r(x)\bigr)=o(r^{n-1}).
$$
In this way we obtain
that \eqref{eq:osmall} holds for $\Haus{n-1}$-a.e. $x\in\Gamma_i\cap\partial^* E$ and
the statement is proved, since $i$ is arbitrary and
$\Gamma_i$ is a $C^1$ hypersurface.
\end{proof}

We can now prove
the missing part $\J_-(E,\Omega)\geq \Per(E,\Omega)/2$ of Theorem~\ref{thm:limper}. 
Let $S\subset\Omega$ be the set where the $\liminf$ in \eqref{eq:finlim} is greater or equal than $1/2$, and notice that
Lemma~\ref{lem:3} shows that $|D\uno_E|\res\Omega$ is concentrated on $S$, so that $\Haus{n-1}(S)\geq
\Per(E,\Omega)$.  If $2\J_-(E,\cdot)=:\mu(\cdot)$ were a $\sigma$-additive measure,  then the well-known implication
\begin{equation}\label{eq:dlb}
\liminf_{r\to 0^+}\frac{\mu(B_r(x))}{\omega_{n-1}r^{n-1}}\geq 1\,\,\,\forall\, x\in S
\quad\Longrightarrow\quad \mu(\Omega)\geq \Haus{n-1}(S)
\end{equation}
would provide us with the needed inequality (see for instance \cite[Theorem~2.56]{afp} for a proof
of \eqref{eq:dlb}). 
However, the traditional proof of \eqref{eq:dlb} works also when $\mu$ is only a superadditive set function defined
on open sets, as we illustrate below.
In particular \eqref{eq:dlb} is applicable to $2\J_-(E,\cdot)$ in view of \eqref{eq:superadd-bis},
which concludes the proof of Theorem~\ref{thm:limper}.\hfill$\square$

\bigskip

We now give a sketch of proof of \eqref{eq:dlb} in the superadditive case, writing $k=n-1$ for convenience.
We can assume without loss of generality $S\Subset\Omega$.  To prove (\ref{eq:dlb}) we fix $\delta\in (0,1)$
and consider all the open balls $\mathcal C$ centered at points of $L$ and with diameter
$d_{\mathcal C}$ strictly less than $\delta$, such that $\mu(\mathcal C)\geq (1-\delta)
\omega_k d_{\mathcal C}^k/2^k$.  By applying Besicovitch covering theorem (see for instance
 \cite[Theorem~2.17]{afp}) we obtain families $\mathcal F_1,\ldots,\mathcal F_\xi$ (with $\xi=\xi(n)$
dimensional constant) with the following properties:
\begin{itemize}
\item[(a)] each family $\mathcal F_i$, $1\leq i\leq\xi$, is disjoint;
\item[(b)] $\cup_1^\xi\bigcup\mathcal F_i$ contains $S$.
\end{itemize}
In particular, using the superadditivity of $\mu$ 
we can estimate from above the pre-Hausdorff measure $\Haus{k}_\delta(L)$ as follows:
$$
\Haus{k}_\delta(S)\leq \sum_{i=1}^\xi\sum_{\mathcal C\in\mathcal F_i} {\omega_k\over 2^k} d_{\mathcal C}^k
\leq \frac{1}{1-\delta}
 \sum_{i=1}^\xi\sum_{\mathcal C\in\mathcal F_i} \mu(\mathcal C)\leq
{\xi\over 1-\delta}\mu(\Omega).$$
By letting $\delta\downarrow 0$ we obtain that $\Haus{k}(S)\leq\xi\,\mu(\Omega)<\infty$.  Using this
information we can improve the estimate, now applying Besicovitch--Vitali covering theorem to the above
mentioned fine cover of $S$, to obtain a disjoint family $\{\mathcal C_i\}$ which covers $\Haus{k}$-almost all 
(hence $\Haus{k}_\delta$-almost all) of $S$.  As a
consequence
$$
\Haus{k}_\delta(S)\leq \sum_i {\omega_k\over 2^k}d_{\mathcal C_i}^k
\leq\sum_i{1\over 1-\delta}\mu(\mathcal C_i)\leq{1\over 1-\delta}\mu(\Omega).
$$
Letting $\delta\downarrow 0$ we finally obtain
$\Haus{k}(S)\leq\mu(\Omega)$, as desired. \hfill$\square$

\subsection{Proof of Theorem~\ref{thm:mainintro} in the case $n=1$.}  Note that 
\begin{equation}\label{defIepsn=1}
\I_\epsilon(\uno_A):=\sup_{I}
\mean{I}\Bigl|\uno_A(x)-\mean{I} \uno_A\Bigr|\, dx,
\end{equation}
and $I$ runs among all intervals with length $\epsilon$. Recall (see for instance \cite{afp}) that in
the 1-dimensional case any set of finite and positive perimeter is equivalent to a finite disjoint union of closed
intervals or half-lines, and the perimeter is the number of the endpoints; in addition $\Per(A)=0$ if and only
if either $|A|=0$ or $|\R\setminus A|=0$.

The inequalities $\I_\epsilon(\uno_A)\leq 1/2$ and $I_\epsilon(\uno_A)\leq\Per(A)$ follow
by \eqref{eq:trivial} and \eqref{eq:Hadwiger} respectively, as in the case $n>1$, hence
$2\limsup_\epsilon\I_\epsilon(\uno_A)\leq\min\{1,\Per(A)\}$. It remains to prove
$$
2\liminf_{\epsilon\to 0} \I_\epsilon(\uno_A)\geq\min\{1,\Per(A)\}
$$
and, since $\Per(A)$ is always a natural number (possibly infinite), we need only to show that
$\Per(A)\geq 1$ implies $2\liminf_\epsilon \I_\epsilon(\uno_A)\geq 1$.
We will prove the stronger implication
\begin{equation}\label{eq:june9}
\Per(A)>0 \quad\Longrightarrow\quad 2\liminf_{\epsilon\to 0} \I_\epsilon(\uno_A)\geq 1.
\end{equation}
To prove \eqref{eq:june9}, notice that $\Per(A)>0$ implies that both $A$ and $\R\setminus A$ have nontrivial measure, so there exist distinct points
$x_1,\,x_2\in\R$ such that $x_1$ is a density point of $A$ and $x_2$ is a density point of
$\R\setminus A$. Hence, for $\epsilon$ sufficiently small we have then
$$
\mean{(x_1-\epsilon/2,x_1+\epsilon/2)} \uno_A>\frac 12
\qquad\text{and}\qquad
\mean{(x_2-\epsilon/2,x_2+\epsilon/2)} \uno_A<\frac 12
$$
We can then use a continuity argument to find, for $\epsilon$ sufficiently small,
a point $x_\epsilon$ such that 
$$
\mean{(x_\epsilon+\epsilon/2,x_\epsilon+\epsilon/2)} \uno_A=\frac 12,
$$
so that $$\mean{(x_\epsilon-\epsilon/2,x_\epsilon+\epsilon/2)}\Bigl|\uno_A(x)-\mean{(x_\epsilon-\epsilon/2,x_\epsilon+\epsilon/2)}\uno_A\Bigr|\,dx=\frac 12.$$
Hence $\Per(A)>0$ implies $2\I_\epsilon(\uno_A)\geq 1$ for $\epsilon$ small enough,
so in particular $2\liminf_\epsilon \I_\epsilon(\uno_A)\geq 1$, as desired.\hfill$\square$

\section{Variants}
\subsection{A localized version of Theorem \ref{thm:mainintro}}
\label{sect:XX}
Let $\Omega\subset \R^n$ be a bounded domain with Lipschitz boundary
and consider the quantity 
\begin{equation}\label{defIeps omega}
\I_\epsilon(f,\Omega):=\epsilon^{n-1}\sup_{{\mathcal F}_\epsilon}\sum_{Q'\in\mathcal F_\epsilon}
\mean{Q'}\Bigl|f(x)-\mean{Q'} f\Bigr|\, dx,
\end{equation}
where ${\mathcal F}_\epsilon$ denotes a collection of disjoint $\epsilon$-cubes $Q'\subset\Omega$
with \emph{arbitrary} orientation and cardinality not exceeding $\epsilon^{1-n}$. 

In analogy with Theorem~\ref{thm:mainintro}, we can also prove the following result.

\begin{theorem}
\label{thm:per}For any measurable set $A\subset\R^n$ one has
$$
\lim_{\epsilon\to 0}\I_\epsilon(\uno_A,\Omega)=\frac 12\min\bigl\{1,\Per(A,\Omega)\bigr\}.
$$
\end{theorem}

\begin{proof}
We begin by noticing that both the upper and the lower bound in the rectifiable case are local, so that 
parts of the proof go throughout without any essential modification.
Hence, we only need to discuss the lower bound in the non-rectifiable case.

If $\Per(A,\Omega)=\infty$, we can find an open smooth subset $\Omega' \Subset \Omega$
such that $\Per(A,\Omega')$ is arbitrarily large (the largeness will be fixed later).
We set $c_0=2^{-n-1}$, consider $c_1$ be given by Lemma~\ref{lem2}, and set $K:=1/c_1$. 
Then, by looking at the proof of Lemma~\ref{lem1} 
it is immediate to check that the same result still holds
with $K=1/c_1$ and considering only $\delta$-cubes 
which intersect $\Omega'$ (this ensures that, if $\delta$ is sufficiently small, all cubes are contained inside $\Omega$) provided $\Per(A;\Omega')>2^{2n+2}n K$.
Thanks to this fact, the proof at the end of Section \ref{sect:unrectif} now goes through without modifications: first we apply 
Lemma~\ref{lem1} to find a
disjoint family ${\mathcal U}_\delta$ of $\delta$-cubes intersecting $\Omega'$ with $\#\,{\mathcal U}_\delta>c_1^{-1}\delta^{1-n}$ and
$$
c_0<\frac{|Q'\cap A|}{|Q'|}< 1-c_0\qquad\text{for all $Q'\in\mathcal U_\delta$,}
$$
and then we apply Lemma~\ref{lem2} with $\epsilon \ll \delta$ to obtain, for each $Q'\in \mathcal U_\delta$, a disjoint family
${\mathcal G}_{\epsilon}(Q')$ of $\epsilon$-cubes satisfying
$$
\frac{|V\cap A|}{|V|}=\frac 12\qquad\forall\, V\in{\mathcal G}_{\epsilon}(Q'),
$$
$$
\#\,{\mathcal G}_{\epsilon}(Q')>c_1\biggl(\frac{\delta}{\epsilon}\biggr)^{n-1}.
$$
We then conclude as in Section~\ref{sect:unrectif}.
\end{proof}

Next, we return to the quantity $[f]$ defined in \cite{bbm}.
As announced in the introduction, we establish the following result:
\begin{corollary}
\label{cor:bbm}
For any measurable set $A\subset Q$ one has
$$
[\uno_A]\leq \frac 12\min\bigl\{1,\Per(A,Q)\bigr\} \leq C\,[\uno_A].
$$
\end{corollary}
\begin{proof}
The result is an immediate consequence of Theorem~\ref{thm:per}
and Lemma~\ref{lem:final} below, where we compare 
the quantities $\I_\epsilon(f)$ with their anisotropic counterparts $[f]_\epsilon$ as
defined in \cite{bbm}.
\end{proof}

\begin{lemma}\label{lem:final} For $\epsilon>0$ and $f\in L^1(Q)$ measurable, let $\I_\epsilon(f,Q)$ 
be defined as in \eqref{defIeps} with cubes $Q'\subset Q$ and let $[f]_\epsilon$ be
defined by \eqref{def[]eps}. Then
\begin{equation}\label{eq:may16}
[f]_\epsilon\leq \I_\epsilon (f,Q)\leq C\, [f]_{\sqrt{n}\epsilon}
\qquad \forall\,\epsilon \in (0,1/\sqrt{n}). 
\end{equation}
\end{lemma}
In the proof of this result, we shall to use the elementary inequalities
\begin{equation}\label{eq:characteristic}
\mean{Q}\Bigl|f(x)-\mean{Q} f\Bigr|\,dx\leq\mean{Q}\mean{Q}|f(x)-f(y)|\,dx\,dy\leq
2 \mean{Q}\Bigl|f(x)-\mean{Q} f\bigr|\,dx.
\end{equation}

\begin{proof} The first inequality in \eqref{eq:may16} is obvious. 
In order to prove the second one, let $\mathcal F_\epsilon=\{Q_i\}_{i\in I}$ be a disjoint family of $\epsilon$-cubes in $Q$
with cardinality of $I$ less than $\epsilon^{1-n}$. For each cube $Q_i$ in $\mathcal F_\epsilon$ we can find a $\sqrt{n}\epsilon$-cube $Q_i'$
containing $Q_i$, contained in $Q$, and with sides parallel to the coordinate axes.  Since $\mathcal F_\epsilon$ is disjoint, the family of the corresponding cubes $Q_i'$ has bounded
overlap, more precisely for each cube $Q_i'$ the cardinality of the set $\{j:\ Q_j'\cap Q_i'\neq\emptyset\}$ does not exceed $c_n$. 
Hence, by an exhaustion procedure, we can partition the index set $I$ in families $I_1,\ldots,I_N$, with $N\leq c_n$, in such a way that
the families
$$
\mathcal G_j':=\left\{Q_i':\ i\in I_j\right\}\qquad j=1,\ldots,N
$$
are disjoint. Since $\#\,\mathcal G_j'\leq \epsilon^{1-n}$ for each $j=1,\ldots,N$, splitting the family $\mathcal G_j'$
in at most $n^{(n-1)/2}+1$ subfamilies with cardinality less than $(\sqrt{n}\epsilon)^{1-n}$, we have
$$
(\sqrt{n}\epsilon)^{1-n}\sum_{Q'\in\mathcal G_j'}\mean{Q'}\Bigl|f(x)-\mean{Q'}f\Bigr|\,dx\leq
\bigl(n^{(n-1)/2}+1\bigr)[f]_{\sqrt{n}\epsilon}.
$$
On the other hand, if $\mathcal G_j$, $j=1,\ldots,N$, denote the corresponding families of original cubes, 
since
$$
\mean{Q}\mean{Q}|f(x)-f(y)|\,dx\,dy \leq (\sqrt{n})^{2n}
\mean{Q'}\mean{Q'}|f(x)-f(y)|\,dx\,dy,
$$
using
\eqref{eq:characteristic} we readily obtain
$$
\sum_{Q\in\mathcal G_j}\mean{Q}\Bigl|f(x)-\mean{Q}f\Bigr|\,dx\leq 2n^n
\sum_{Q'\in\mathcal G_j'}\mean{Q'}\Bigl|f(x)-\mean{Q'}f\bigr|\,dx.
$$
Hence, since $\mathcal F_\epsilon=\mathcal G_1\cup\cdots\cup\mathcal G_N$, 
adding with respect to $j$ and using the fact that $\mathcal F_\epsilon$ is arbitrary, we obtain the second inequality in \eqref{eq:may16} with 
$$C:=2c_n(n^{(n-1)/2}+1)n^{(3n-1)/2}.$$
\end{proof}

\begin{remark}
{\rm
Notice that Corollary~\ref{cor:bbm} could be refined by adapting the argument used in 
Section~\ref{sec:lbound} to the setting of \cite{bbm} where the cubes are forced to be parallel to the coordinate axes: more precisely, if we denote by $K_n$ the largest possible intersection of a hyperplane with
the unit cube,
i.e.,
$$
K_n:=\sup_{H\subset \R^n\, \text{hyperplane}}\Haus{n-1}(H\cap Q),
$$
then
$$
[\uno_A] \geq \frac12 \min \Bigl\{1, \frac{1}{K_n}\,\Per(A,Q)\Bigr\}.
$$
}
\end{remark}

\subsection{A new characterization of the perimeter}
\label{sect:no constraint}

Theorem \ref{thm:mainintro} provides a characterization of sets of finite perimeter only when $\lim_{\epsilon \to 0}\I_\epsilon(\uno_A)<1/2$.
 This critical 
 threshold could be easily tuned by modifying the upper bound on the cardinality
of the families $\mathcal F_\epsilon$ in \eqref{defIeps}, as \eqref{eq:june3} shows. 
As a consequence a byproduct of our results is a characterization
of the perimeter free of truncations:
\begin{equation}
\label{eq:per}
\lim_{\epsilon\downarrow 0}
\sup_{{\mathcal H}_\epsilon}\epsilon^{n-1}\sum_{Q'\in\mathcal H_\epsilon}
2\mean{Q'}\Bigl|\uno_A(x)-\mean{Q'} \uno_A\Bigr|\, dx=\Per(A),
\end{equation}
where now ${\mathcal H}_\epsilon$ 
denotes a collection of disjoint $\epsilon$-cubes $Q'\subset \R^n$ 
with arbitrary orientation but no constraint on cardinality.

In order to prove \eqref{eq:per} we notice that the argument in Section \ref{sec:upperbound}, based on the
relative isoperimetric inequality, easily gives
$$
\epsilon^{n-1}\sum_{Q'\in\mathcal H_\epsilon}
2\mean{Q'}\Bigl|\uno_A(x)-\mean{Q'} \uno_A\Bigr|\, dx\leq \Per(A).
$$
On the other hand, we can use \eqref{eq:june3} to get
$$
\liminf_{\epsilon\downarrow 0}
\sup_{{\mathcal H}_\epsilon}\epsilon^{n-1}\sum_{Q'\in\mathcal H_\epsilon}
2\mean{Q'}\Bigl|\uno_A(x)-\mean{Q'} \uno_A\Bigr|\, dx \geq \sup_{M>0} \min\{M,\Per(A)\} =\Per(A),
$$
proving \eqref{eq:per}.

Notice that the formulation given in Theorem~\ref{thm:mainintro} is stronger than \eqref{eq:per}, because it shows that a cardinality
constrained maximization is sufficient to provide finiteness of perimeter, under the critical threshold.

It is also worth noticing that \eqref{eq:per} can be extended to general $\Z$-valued functions:
indeed, the argument in Section \ref{sec:lbound} can be easily adapted to prove that
if $f \in BV(\R^n;\Z)$ then
$$
\liminf_{r \to 0}\frac{\J_-(f,B_r(x))}{|Df|(B_r(x))} \geq \frac12
\qquad \text{for $|Df|$-a.e. $x$,}
$$
where $\J_-(f;B_r(x))$ is defined 
analogously to $\J_-(E;B_r(x))$ in Section \ref{sec:lbound}, thus showing that
$$
\liminf_{\epsilon\downarrow 0}
\sup_{{\mathcal H}_\epsilon}\epsilon^{n-1}\sum_{Q'\in\mathcal H_\epsilon}
2\mean{Q'}\Bigl|f(x)-\mean{Q'} f\Bigr|\, dx\geq |Df|(\R^n),
$$
while the converse inequality follows by writing
$$
f=\sum_{k>0} \uno_{\{f \geq k\}} -\sum_{k>0}\uno_{\{f \leq - k\}},
$$
which gives
\begin{align*}
&\sup_{{\mathcal H}_\epsilon}\epsilon^{n-1}\sum_{Q'\in\mathcal H_\epsilon}
2\mean{Q'}\Bigl|f(x)-\mean{Q'}f\Bigr|\,dx\\
&\leq\sum_{k>0}
\sup_{{\mathcal H}_\epsilon}\epsilon^{n-1}\sum_{Q'\in\mathcal H_\epsilon}
2\mean{Q'}\Bigl|\uno_{\{f \geq k\}}(x)-\mean{Q'}\uno_{\{f \geq k\}}\Bigr|\,dx\\
&\qquad+\sum_{k>0}
\sup_{{\mathcal H}_\epsilon}\epsilon^{n-1}\sum_{Q'\in\mathcal H_\epsilon}
2\mean{Q'}\Bigl|\uno_{\{f \leq -k\}}(x)-\mean{Q'}\uno_{\{f \leq -k\}}\Bigr|\,dx\\
&\leq \sum_{k>0}\Per ({\{f \geq k\}})+\sum_{k>0}\Per ({\{f \leq -k\}})
=|Df|(\R^n).
\end{align*}

These facts provide a new  characterization both of sets of finite perimeter and of the perimeter of sets, independent of the theory of
distributions. Heuristically, given $\epsilon>0$, any maximizing family ${\mathcal F}_\epsilon$ provides a sort of boundary on scale
$\epsilon$ of $A$, and some proofs (in particular the one of Lemma~\ref{lem1}, see also Remark~\ref{rem:simple_idea}) 
make more rigorous this idea.

It is interesting also to
compare this result with another non-distributional characterization of sets of finite perimeter due to H. Federer, see \cite[Theorem~4.5.11]{Fed}: $\Per(A)$
is finite if and only if the essential boundary $\partial^* A$, namely the set in \eqref{eq:defbf} of points of density neither 0 nor 1,
has finite $\Haus{n-1}$-measure, and then $\Per(A)=\Haus{n-1}(\partial^*A)$. However, Federer's characterization and the one provided
by this paper seem to be quite different.

\subsection{Approximation of the total variation}
\label{sect:no constraint2}

Motivated by the results in Section~\ref{sect:no constraint}, for $f\in L^1_{\rm loc}(\R^n)$ we may define
$$
{\sf K}_\epsilon(f):=\sup_{{\mathcal H}_\epsilon}\epsilon^{n-1}\sum_{Q'\in\mathcal H_\epsilon}
2\mean{Q'}\Bigl| f(x)-\mean{Q'} f\Bigr|\, dx
$$
where, once more, ${\mathcal H}_\epsilon$
denotes a collection of disjoint $\epsilon$-cubes $Q'\subset\R^n$ 
with arbitrary orientation but no constraint on cardinality. Then, the result of the previous section can be read
as follows:
$$
\lim_{\epsilon\to 0}{\sf K}_\epsilon(f)=|Df|(\R^n)
$$
for any $\Z$-valued function $f$. 

For general $BV_{\rm loc}$ functions $f$, the asymptotic analysis of ${\sf K}_\epsilon$ seems to be more difficult
to grasp. By considering smooth functions and functions with a jump discontinuity along a hyperplane, one is led to the conjecture that
\begin{equation}\label{eq:june3bis}
\lim_{\epsilon\to 0}{\sf K}_\epsilon(f)=\frac{1}{4}|D^a f|(\R^n)+\frac{1}{2}|D^s f|(\R^n)
\end{equation}
for all $f\in SBV_{\rm loc}(\R^n)$, where (see \cite{afp}) $SBV_{\rm loc}(\Omega)$ is the vector space of all
$f\in BV_{\rm loc}(\Omega)$ whose distributional derivative is the sum of a measure
$D^a f$ absolutely continuous w.r.t. $\Leb{n}$ and a measure $D^s f$ concentrated on a set $\sigma$-finite w.r.t. $\Haus{n-1}$.
For functions $f\in BV_{\rm loc}\setminus SBV_{\rm loc}$, having the so-called Cantor part of the derivative, it might possibly happen that
${\sf K}_\epsilon(f)$ oscillates as $\epsilon\to 0$ between $\frac 14 |Df|(\R^n)$ and $\frac 12 |Df|(\R^n)$.

Notice that all functionals ${\sf K}_\epsilon$ are $L^1_{\rm loc}(\R^n)$-lower semicontinuous.
On the other hand, it is natural to expect that the $\Gamma$-limit of ${\sf K}_\epsilon$ w.r.t. the $L^1_{\rm loc}(\R^n)$ topology exists and
that %it is equal to the lower semicontinuous relaxation w.r.t. the $L^1_{\rm loc}(\R^n)$ topology of the right hand side in \eqref{eq:june3bis}, namely
$$
\Gamma-\lim_{\epsilon\to 0}{\sf K}_\epsilon(f)=\frac{1}{4}|Df|(\R^n).
$$

\section{Appendix: proof of \eqref{eq:Hadwiger}} 

In this section we prove the relative isoperimetric inequality in the cube, in the sharp form provided by \eqref{eq:Hadwiger},
in any Euclidean space $\R^n$, $n\geq 1$. To this aim,
we introduce the Gaussian isoperimetric function $I:(0,1)\to (0,1/\sqrt{2\pi}]$ defined by
$$
I(t):=\varphi\circ\Phi^{-1}(t)\quad\text{with}\quad \varphi(x)=\frac{e^{-x^2/2}}{\sqrt{2\pi}},\,\,\,\,\Phi(x):=\int_{-\infty}^x \varphi(y)\, dy,\,\,\,\,x\in\R.
$$
We extend $I$ by continuity to $[0,1]$ setting $I(0)=I(1)=0$. Notice that $I(1/2)=\varphi(0)=1/\sqrt{2\pi}$ and it
is also easy to check that $I(t)=I(1-t)$.

\begin{lemma} \label{calculus} The function $K(t):=\sqrt{2\pi} I(t)-4t(1-t)$ is nonnegative in $[0,1]$ and $K(t)=0$ if and
only if $t\in\{0,1/2,1\}$.
\end{lemma}
\begin{proof} Let us record an additional property of $I$:
\begin{equation}\label{eq:smooth}
I\in C^\infty(0,1)\quad\text{and}\quad I''=-\frac{1}{I}\,\,\text{on $(0,1)$.}
\end{equation}
Indeed, from $I\circ\Phi=\varphi$ and $\Phi'=\phi$ we obtain $(I'\circ\Phi(x))\varphi(x)=\varphi'(x)=-x\varphi(x)$, so that
$I'\circ\Phi(x)=-x$. By differentiating once more we get $I''(t)=-1/\Phi'\circ\Phi^{-1}(t)=-1/I(t)$ in $(0,1)$,
as desired.

Since $I$ attains its maximum at $1/2$, from \eqref{eq:smooth} we obtain that $I'>0$ in $(0,1/2)$, hence
there exists a unique $t_0\in (0,1/2)$ such that $I(t_0)=\sqrt{2\pi}/8$. 

Since $K$ vanishes on $\{0,1/2,1\}$ and it inherits from $I$
the symmetry property $K(t)=K(1-t)$, it suffices to check that $K$ is strictly positive in $(0,1/2)$.
Differentiating $K$ we get
\begin{equation}
K'(t)=\sqrt{2\pi}I'(t)-4+8t,
\end{equation}
In particular, $K'(1/2)=I'(1/2)=0$. Differentiating once more 
and using \eqref{eq:smooth} we get
\begin{equation}
K''(t)=\sqrt{2\pi}I''(t)+8=-\frac{\sqrt{2\pi}}{I(t)}+8,
\end{equation}
so that 
\begin{equation}\label{eq:conca}
K''<0\quad\text{in $(0,t_0)$},\qquad
K''>0\quad\text{in $(t_0,1/2)$.}
\end{equation}
In particular $K'(1/2)=0$ gives $K'<0$ in $[t_0,1/2)$ and therefore $K(1/2)=0$ gives 
\begin{equation}\label{eq:final1}
K>0\quad\text{in $[t_0,1/2)$.}
\end{equation}
For the interval $[0,t_0]$ we use $K(t_0)>0$, $K(0)=0$ and the concavity of $K$ in $(0,t_0)$, ensured by \eqref{eq:conca},
to get
\begin{equation}\label{eq:final2}
K>0\quad\text{in $(0,t_0]$.}
\end{equation}

The conclusion follows by \eqref{eq:final1} and \eqref{eq:final2}.
\end{proof}

Now, let us prove \eqref{eq:Hadwiger}. Combining \cite[Proposition 5 and Theorem 7]{bm} we obtain the
inequality
\begin{equation}\label{eq:13may}
I\biggl(\int_{(0,1)^n} f\,dx\biggr)\leq \int_{(0,1)^n}\bigl[I(f)+\frac{1}{\sqrt{2\pi}}|\nabla f|\bigr]\, dx
\end{equation}
for any locally Lipschitz function $f:(0,1)^n\to [0,1]$. 
Then, the $BV$ version of the Meyers-Serrin approximation theorem
(due to Anzellotti-Giaquinta, see for instance \cite[Theorem~3.9]{afp}) enables
us to approximate  in $L^1((0,1)^n)$ any function $f\in BV((0,1)^n)$ by
functions $f_k\in C^\infty((0,1)^n)$ in such a way that 
$$
\lim_{k\to\infty}\int_{(0,1)^n}|\nabla f_k|\,dx=|Df|\bigl((0,1)^n\bigr).
$$
In addition, if $f:(0,1)^n\to [0,1]$, a simple truncation argument provides approximating
functions $f_k$ with the same property. It then follows from \eqref{eq:13may} that
$$
I\biggl(\int_{(0,1)^n} f\,dx\biggr)\leq \int_{(0,1)^n}I(f)\,dx+\frac{1}{\sqrt{2\pi}}|Df|\bigl((0,1)^n\bigr)
$$
for all $f:(0,1)^n\to [0,1]$ with bounded variation. Since $I(0)=I(1)=0$, choosing $f=\uno_E$ gives
$$
I(|E|)\leq \frac{1}{\sqrt{2\pi}}\Per\bigl(E,(0,1)^n\bigr).
$$
We conclude using Lemma~\ref{calculus}.
\quad

\smallskip
\noindent
{\bf Acknowledgements.} The first author (LA) was partially supported by the ERC ADG project GeMeThNES, the second
author (JB) was partially supported by NSF grant DMS-1301619, the third author (HB) 
was partially supported by NSF grant DMS-1207793
and by grant number 238702 of the European Commission (ITN, project FIRST), the fourth author (AF) was partially supported by
NSF grant DMS-1262411.

The authors thank B. Kawohl and F. Barthe for very useful information about the relative isoperimetric inequality \eqref{eq:Hadwiger}.


\begin{thebibliography}{99}

\bibitem{afp} {\sc L.Ambrosio, N.Fusco, D.Pallara:} {\em Functions of bounded variation and free discontinuity
problems.} Oxford University Press, 2000. 

%\bibitem{anp} {\sc L.Ambrosio, M.Novaga, E.Paolini:} {\em Some regularity results for minimal crystals.}
%ESAIM COCV, {\bf 8} (2002), 69--103.

\bibitem{bl} {\sc D.Bakry, M.Ledoux:} {\em L\'evy-Gromov isoperimetric inequality for an infinite
dimensional diffusion generator} Invent. Math., {\bf 123} (1996), 259--281.

\bibitem{bm} {\sc F.Barthe, B.Maurey:} {\em Some remarks on isoperimetry of Gaussian type.}
Ann. Inst. Henri Poincar\'e, Probabilit\'es et Statistiques, {\bf 36} (2000), 419--434.

\bibitem{bo1} {\sc S.G.Bobkov:} {\em A functional form of the isoperimetric inequality for the Gaussian measure.} J. Funct. Anal., {\bf 135} (1996), 39--49.

\bibitem{bo2} {\sc S.G.Bobkov:} {\em An isoperimetric inequality on the discrete cube, and an elementary
proof of the isoperimetric inequality in Gauss space.} Ann. of Probab., {\bf 25} (1997), 206--214.

\bibitem{bbm} {\sc J.Bourgain, H.Brezis, P.Mironescu:} {\em A new function space and applications.}
Preprint, 2014.

\bibitem{X1} {\sc J.Bourgain, H.Brezis, P.Mironescu:} {\em Another look at Sobolev spaces.}  
In ``Optimal Control and Partial Differential Equations'', J.L. Menaldi, E. Rofman et A. Sulem, eds, IOS Press, 2001, 439--455.

\bibitem{X2} {\sc H.Brezis}: {\em How to recognize constant functions. Connections with Sobolev Spaces.} 
Uspekhi Mat. Nauk, {\bf 57} (2002), 59--74 (in Russian). English translation in Russian Math. Surveys, {\bf 57} (2002), 693--708.

\bibitem{X3} {\sc J.Davila}:  {\em On an open question about functions of bounded variation.} Cal. Var. PDE, {\bf 15}  (2002), 519--527.

\bibitem{dg1} {\sc E.De Giorgi:} {\em Nuovi teoremi relativi alle misure $(r-1)$-dimensionali in uno
spazio a $r$ dimensioni.} Ricerche Mat., {\bf 4} (1955), 95--113.

%\bibitem{evgar} {\sc L.C.Evans, R.F.Gariepy:}
%{\em Measure Theory and Fine Properties of Functions.} Studies in Advanced Mathematics, 1992.

\bibitem{Fed1} {\sc H.Federer:} {\em A note on the Gauss-Green theorem.} Proc. Amer. Math. Soc., {\bf 9} (1958),
447--451.

\bibitem{Fed} {\sc H.Federer:} {\em Geometric Measure Theory.} Springer-Verlag, 1969.

\bibitem{Ha} {\sc H.Hadwiger:} {\em Gitterperiodische Punktmengen und Isoperimetrie.} Monatsh. Math., {\bf 76}
(1972), 410--418.

\bibitem{JN}
{\sc F.John, L.Nirenberg:}
{\em On functions of bounded mean oscillation.}
Comm. Pure Appl. Math., {\bf 14} (1961), 415-426. 



\end{thebibliography}
\end{document}